\documentclass[11p,reqno]{amsart}
\usepackage{amsmath}
\usepackage{amsfonts}
\usepackage{amssymb}
\usepackage{graphicx}
\usepackage{verbatim}
\usepackage{color}
\usepackage{appendix}
\newtheorem{theorem}{Theorem}
\newtheorem*{theorem*}{Theorem}

\newtheorem{proposition}[theorem]{Proposition}
\newtheorem*{proposition*}{Proposition}

\newtheorem{lemma}[theorem]{Lemma}
\newtheorem*{lemma*}{Lemma}

\newtheorem{definition}[theorem]{Definition}
\newtheorem{example}[theorem]{Example}
\newtheorem*{example*}{Example}
\newtheorem*{definition*}{Definition}

\newtheorem*{cor*}{Corollary}
\newtheorem*{result*}{Result}

\newtheorem*{hypotheses*}{Hypotheses}

\def\Hess{\operatorname{Hess}}

\newcommand{\eps}{{\varepsilon}}

\numberwithin{equation}{section}

\newcommand{\note}[1]{\textbf{\textcolor{blue}{#1}} \\ \\}

\begin{document}
\title[K\"ahler metrics on tube domains]{Positively curved K\"ahler metrics on tube domains and their applications to optimal transport}

\author{Gabriel Khan}\thanks{The research of GK is partially supported by Simon's Collaboration Grant 849022. }
\address{Gabriel Khan. Department of Mathematics, Iowa State University, Ames, IA, 50011}
\email{gkhan@iastate.edu}

\author{Jun Zhang}\thanks{J.Z. was previously supported by a subcontract through UCLA from AFOSR grant FA9550-19-1-0213."}
\address{Jun Zhang. Department of Mathematics, University of Michigan, Ann Arbor, MI 92093, USA}
\email{junz@umich.edu}

\author{Fangyang Zheng}\address{Fangyang Zheng. School of Mathematical Sciences, Chongqing Normal University, Chongqing 401331, China}
\email{20190045@cqnu.edu.cn; franciszheng@yahoo.com} \thanks{F.Z. is partially supported by National Natural Science Foundations of China
with the grant No.12141101 and 12071050, Chongqing grant cstc2021ycjh-bgzxm0139, and is supported by the 111 Project D21024.}


\begin{abstract} In this article, we study a class of K\"ahler manifolds defined on tube domains in $\mathbb{C}^n$, and in particular those which have $O(n) \times \mathbb{R}^n$ symmetry. For these, we prove a uniqueness result showing that any such manifold which is complete and has non-negative orthogonal bisectional curvature ($n \geq 3$) or non-negative bisectional curvature ($n \geq 2$) is biholomorphically isometric to $\mathbb{C}^n$. 

We also consider another curvature tensor called the \emph{orthogonal anti-bisectional curvature}. We find necessary and sufficient conditions for a complete $O(n)$-symmetric tube domain to have non-negative orthogonal anti-bisectional curvature and provide several examples of complete metrics which satisfy this condition. Finally, we discuss some applications of these spaces within optimal transport. In particular, we study ``synthetic" curvature bounds for non-smooth geometries and how they can be applied to the rough geometry induced by the Monge cost $c(x,y)=\|x-y\|$.

\end{abstract}

\maketitle

\section{Introduction}

 In this paper, we study K\"ahler metrics defined on tube domains, which are domains of the form \[ T \Omega = \left \{ z \in \mathbb{C}^n ~|~ z = x+iy, x \in \Omega \right \} \] where $\Omega$ is a convex domain in $\mathbb{R}^n$.
In particular, we focus on metrics $\omega_\Psi$ whose K\"ahler potential $\Psi$ (in the $z$-coordinates) is independent of $y$, and so admit a natural translation symmetry. 
Our primary focus in this paper are K\"ahler metrics which are complete and whose curvature is non-negative (in several different senses). In general, positive curvature on K\"ahler manifolds is a very strong assumption which greatly restricts the geometry and topology. On the other hand, positively curved metrics have very interesting geometric properties and in many geometric or analytic applications, an assumption of positive curvature is necessary to establish results. Our results fall in line with both these expectations, showing that positive curvature is quite restrictive but also has interesting applications.

Before discussing our results, it is worth noting that the tube domain $T \Omega$ has a natural interpretation as the tangent bundle of the domain $\Omega$. From this perspective, the $\mathbb{R}^n$-symmetric K\"ahler metric on $T \Omega$ coincides with the so-called \textit{Sasaki metric}\footnote{More generally, the Sasaki metric is an almost-Hermitian metric defined on the tangent bundle of a Riemannian manifold $(M,g)$ with an affine connection $\nabla$. In general, this structure need not be complex or symplectic, but we will only consider cases where it is K\"ahler.} \cite{Dombrowski}. To distinguish this class of spaces from more general K\"ahler manifolds, we will refer to translation-symmetric K\"ahler metrics on tube domains as \textit{K\"ahler Sasaki metrics}.

\subsection{Positively curved $O(n)$-symmetric metrics}

In the first several sections of this paper, we focus on K\"ahler Sasaki metrics which are complete and $O(n)$-symmetric. In other words, we take $\Omega$ to be the ball $B_a\subseteq {\mathbb R}^n$ of radius $a$ (with $a$ possibly infinite) and suppose that rotations of $\Omega$ preserve the K\"ahler metric on $T \Omega$.
In Section \ref{Bisectional}, we study the bisectional and orthogonal bisectional curvatures of these spaces. We find the following uniqueness theorem.

\begin{theorem*} \label{Main}
	Let $\Psi$ be a strictly convex function on the ball $B_a\subseteq {\mathbb R}^n$ of radius $0<a\leq \infty$. Assume that $\Psi(x)$ depends solely on the norm of $x$ (i.e. the metric is rotationally symmetric) and that the associated K\"ahler metric $\mathfrak{h}_{\Psi}$ on $T B_a$ is complete. If either $n= 2$ and the bisectional curvature is everywhere non-negative, or $n\geq 3$ and the orthogonal bisectional curvature is everywhere non-negative, then $a=\infty $ and the metric $\mathfrak{h}_{\Psi}$ is the flat Euclidean metric on ${\mathbb C}^n$.
\end{theorem*}

Put more simply, this shows that in dimension 2 the only K\"ahler Sasaki metrics which is complete, rotationally symmetric, and whose bisectional curvature is non-negative is the standard flat metric on $\mathbb{C}^n$. Furthermore, we have the same conclusion when $n\geq 3$ under the weaker assumption that the orthogonal bisectional curvature is non-negative.

Since this result rules out any (non-trivial) examples with positive bisectional curvature, we turn our attention to other notions of curvature. In Section \ref{antibisectionalsection}, we consider the \textit{anti-bisectional curvature}, defined as follows:
\[ \mathfrak{A} (U,V) = 4 R^h (U, \overline{V}, U, \overline{V}) \]
for two holomorphic vector fields $U$ and $V$.
For K\"ahler Sasaki metrics, there is a notion of ``non-negative anti-bisectional curvature" (abbreviated (NAB)) and  ``non-negative orthogonal anti-bisectional curvature" (abbreviated (NOAB)), which is defined precisely in Section \ref{antibisectionalsection}.\footnote{As a note of caution, (NAB) (respectively (NOAB)) are weaker assumptions than assuming $\mathfrak{A} \geq 0$ for all (respectively all orthogonal) type (1,0)-vectors.}
For $O(n)$-symmetric tube domains, we find several integral-differential inequalities which are equivalent to (NOAB) (see Proposition \ref{NOABforOnsymmetric} for a precise statement). Using this, we find several new examples of such metrics. As a result, (NOAB) does not imply the same type of uniqueness properties as with bisectional curvature.

Following the theme that positively curved metrics have desirable analytic properties, 
recent work by the first two named authors \cite{KhanZhang} found a connection between orthogonal anti-bisectional curvature and optimal transport. More precisely, given a (NOAB) metric on a tube domain, it is possible to construct a cost function $c:X \times Y \to \mathbb{R}$ (with $X,Y \subset \mathbb{R}^n$) satisfying the \textit{(MTW) condition}, which plays a crucial role in the regularity theory of optimal transport
(see Section \ref{Regularityoftransport} for more details).

As such, we can use our examples of (NOAB) K\"ahler Sasaki metrics to find costs whose optimal transport has good regularity properties. For instance, we provide the following example.

\begin{example*}
The cost function $c:\mathbb{R}^n \times \mathbb{R}^n \to \mathbb{R}$ given by $c(x,y)= \|x-y\| - C \log(\|x-y\|)+C $ for ($C>0$) satisfies (MTW).
\end{example*}

For K\"ahler Sasaki metrics which are $O(n)$-symmetric, the induced cost function depends only on the Euclidean distance between $x$ and $y$. As such, by finding $O(n)$-symmetric metrics with (NOAB), we can find cost functions which have a natural geometric interpretation. For $c$ small, this example gives a cost which approximates Monge's original cost function $c(x,y) = \|x-y\|$ but which has better regularity properties.\footnote{The cost $c(x,y) = \sqrt{\epsilon+\|x-y\|^2}$ also satisfies the $MTW$ condition and approximates Monge's cost, but corresponds to a metric which is \emph{incomplete}. }

\subsection{Synthetic curvature of K\"ahler Sasaki metrics}

Having seen how the study of tube domains has applications to optimal transport, we then use ideas from optimal transport to better understand the geometry of K\"ahler Sasaki metrics. In particular, we study \textit{synthetic} curvature bounds when the K\"ahler potential is not $C^4$, (and thus the usual curvature expressions are not well defined). By analogy, these results are similar to the CAT($\kappa$)-inequality, which provides a synthetic version of lower bounds on the sectional curvature.
 Using the connection between (NOAB) and optimal transport, we provide synthetic characterizations of non-negative orthogonal anti-bisectional curvature and non-negative anti-bisectional curvature. 
 
 At first, $O(n)$-symmetric K\"ahler metrics and synthetic curvature conditions might seem to be distinct topics. However, there is a problem in optimal transport which connects the two, which is to establish a regularity theory for optimal transport with respect to the cost $c(x,y)=\|x-y\|$. This was the cost function originally considered by Monge, but it does not satisfy the assumptions of the standard regularity theory of optimal transport. There are some results which suggest that a weaker regularity theory holds for this cost, but the behavior of the solutions are not well understood. One possible route to study this problem is to try to apply techniques from complex geometry (e.g., K\"ahler-Ricci flow) to regularize the cost function and then try to establish regularity for the limiting solutions.

As the Monge potential is singular, we cannot solve the K\"ahler-Ricci flow on the associated tube domain in the classical sense. Instead, we use a formal computation to derive a family of potentials
   \[ \Psi_{Monge}(r,t)= r- 3t\log(r), \]
which can be understood as first-order Taylor expansion for renormalized Ricci flow. This gives a canonical deformation of Monge's potential, which may be useful to study the regularity theory of the associated transport.

\subsection{Organization of the paper}

 Section \ref{Background} provides background information on Hessian manifolds and K¨ahler metrics on tube
domains. In Section \ref{Bisectional}, we discuss the bisectional and
orthogonal bisectional curvatures of $O(n)$-symmetric tube domains. In Section \ref{antibisectionalsection}, we study the anti-bisectional curvature of rotational symmetric tube domains. Section \ref{Regularityoftransport} discusses the relationship between orthogonal anti-bisectional curvature and optimal transport. Section \ref{SyntheticCurvature} studies curvature bounds on tube domains when the K\"ahler potential is not $C^4$. In particular, in Subsection \ref{Monge cost and KR flow} we study the regularity problem of optimal transport with respect to the Monge cost and how ideas from K\"ahler-Ricci flow can be used to study this problem. Finally, there is an Appendix at the end of the paper, which verifies that a particular family of metrics satisfies (NOAB). 

\section{Background on Hessian metrics and K\"ahler metrics on tube domains} \label{Background}

In this section, we provide some background on Hessian manifolds and the K\"ahler metrics on their tube domains. A Riemannian manifold $(\Omega, g)$ is said to be a Hessian manifold if
\begin{enumerate}
    \item $\Omega$ admits a flat connection $D$,
    \item  such that around every point $x \in \Omega$, there is an open neighborbood $U_x \subset \Omega$,
    \item and a function $\Psi: U_x \to \mathbb{R}$, such that
    \[ g = D^2 \, \Psi. \]
\end{enumerate}

For convex domains in Euclidean space, we can construct Hessian metrics at will, simply by choosing a convex function and using the connection induced by differentiation in coordinates.\footnote{Since Hessian manifolds must be affine (i.e. admit a flat connection), it is non-trivial to find compact examples which are not tori.}
More precisely, we consider a convex domain $\Omega \subseteq {\mathbb R}^n$ with standard coordinates $x=(x_1, \ldots , x_n)$ and a strongly convex function $\Psi:\Omega \to \mathbb{R}$. For now, we will assume that $\Psi$ is $C^4$, but in Section \ref{SyntheticCurvature} we consider less smooth potentials.
Using this potential, we define a Riemannian metric $g$ on $\Omega$, which is given (in $x$-coordinates) by
 $$ g = g_{\Psi} := \sum_{i,j=1}^n \Psi_{ij} dx_i \otimes dx_j. $$
 We call this the Hessian metric with potential $\Psi$. Here and below we denote by $\Psi_{ij}$ the partial derivative $\frac{\partial^2 \Psi}{\partial x_i \partial x_j}$.
To define an associated K\"ahler manifold, we first define the tube domain $T \Omega$. To do so, we write $z_i = x_i + \sqrt{-1} y_i$ as the standard holomorphic coordinates of the complex Euclidean space ${\mathbb C}^n$, and consider the domain $ T \Omega \subseteq {\mathbb C}^n$ given by
$$ T \Omega = \{ (x,y) \mid x \in \Omega , y \in {\mathbb R}^n \} = \Omega \times {\mathbb R}^n \subseteq {\mathbb C}^n.$$
The $(1,1)$-form
$$ \omega_{\Psi} = \sqrt{-1} \sum_{i,j=1}^n \Psi_{ij}(x) dz_i \wedge d\overline{z_j} $$
defined on $T \Omega$ is positive and closed, hence gives a K\"ahler metric on $T \Omega$. We will call $h_{\Psi}$ the {\em K\"ahler Sasaki metric} with potential $\Psi$. As mentioned previously, the reason for this terminology is that it coincides with the almost-Hermitian structure on the tangent bundle of a Riemannian manifold with an affine connection, better known as the \textit{Sasaki metric} \cite{Dombrowski}. Here, the connection $\nabla$ is suppressed, but corresponds to differentiation with respect to the $x$-coordinates (i.e. in the x-coordinates, the Christoffel symbols of $\nabla$ vanish). For simplicity, we will often denote the metric by $\mathfrak{h}$ and  K\"ahler form by $\omega$ (dropping the subscripts from each).

\subsection{Completeness and Curvature}

We are primarily interested in \textit{complete} metrics, so it is necessary to find a condition which ensures that the K\"ahler Sasaki metric is complete. To this end, we make the following observation, which was originally proved by Molitor \cite{Molitor}.

\begin{lemma}
The K\"ahler manifold $(T \Omega, \mathfrak{h}_{\Psi})$ is complete if and only if the Riemannian manifold $(\Omega , g_{\Psi})$ is complete.
\end{lemma}

\begin{proof}
It suffices to prove the `if' direction. Let $\gamma (s)= (x(s), y(s))$ be a smooth curve in $T \Omega=\Omega \times {\mathbb R}^n$ going to infinity. We have
$$ |\dot \gamma(s)|^2_{\mathfrak{h}} = \sum_{i,j=1}^n \left(\dot x_i \dot x_j \Psi_{ij}(x) + \dot y_i \dot y_j \Psi_{ij}(x) \right) . $$
Clearly, $|\dot \gamma(s)|_\mathfrak{h} \geq |x'(s)|_g$, so if the curve $x(s)$ goes to the boundary of $\Omega$, then $x(s)$, hence $\gamma(s)$, is of infinite length as $g$ is complete.  If, on the other hand, the curve $x(s)$ is contained in a compact subset $K$ of $\Omega$, then $y(s)$ must tend to infinity. Take $\eps > 0$ such that $g\geq \eps\, g_0$ on $K$, where $g_0$ is the Euclidean metric of ${\mathbb R}^n$, we have $| \dot \gamma(s) |_\mathfrak{h} \geq |\dot y(s)|_\mathfrak{h} \geq \eps\, |\dot y(s)|_{g_0}$. Hence the length of $\gamma (s)$ is again infinite. This completes the proof of the lemma.
\end{proof}

We now turn our attention to the curvature of K\"ahler Sasaki metrics.
 For brevity, we will not derive the full expressions for the curvature, which can be found in full detail in Satoh \cite{Satoh}. Following the convention of \cite{KhanZhang}, we denote $\Psi_{ijk}=\frac{\partial^3 \Psi}{\partial x_i\partial x_j \partial x_k }$, etc., and use $\Psi^{ij}$ to denote the elements of the matrix inverse to $\Psi_{ij}$. For the K\"ahler manifold $(T \Omega, \mathfrak{h})$, under the natural frame of the holomorphic coordinate $z$, we have $\frac{\partial} {\partial z_i} = \frac{1}{2}( \frac{\partial} {\partial x_i} - \frac{\partial} {\partial y_i} )$, so we have
$$ \mathfrak{h}_{i\overline{j}} = \Psi_{ij}, \ \ \ \ \mathfrak{h}_{i\overline{j},k} = \frac{1}{2} \Psi_{ijk}, \ \  \ \ \mathfrak{h}_{i\overline{j},k\overline{\ell}} = \frac{1}{4} \Psi_{ijk\ell }, \ \ \ \ \mathfrak{h}^{\overline{i}j} = \Psi^{ij}. $$

The components of the curvature tensor $R^{(\mathfrak{h})}$ of $\mathfrak{h}$ are given by
\begin{equation}
R^{(\mathfrak{h})}_{i\overline{j}k\overline{\ell}} = - \frac{1}{4}\Psi_{ijk\ell} + \sum_{p,q} \frac{1}{4} \Psi_{ipk}\, \Psi_{jq\ell} \,\Psi^{pq}
\end{equation}
In particular, for tangent vectors $u=\sum u_i\frac{\partial} {\partial z_i}$ and $v=\sum v_i\frac{\partial} {\partial z_i}$, the bisectional curvature of $h$ is given by
\begin{equation}
R^{(\mathfrak{h})}_{u\overline{u}v\overline{v}} = \sum_{p,q} - \frac{1}{4}\Psi_{u\overline{u}v\overline{v}} + \frac{1}{4} \Psi_{uvp}\, \overline{\Psi_{uvq}} \,\Psi^{pq}.
\end{equation}
Here we denoted by $\Psi_{uvp}= \sum_{i,j}  u_i v_j \Psi_{ijp}$, etc. for the sake of simplicity.

\subsection{$O(n)$-symmetric Hessian metrics}

We now specialize our attention to the case when $\Psi$ is \textit{rotationally symmetric}.
That is to say, we take $\Omega$ to be the ball $B_a\subseteq {\mathbb R}^n$ of radius $a$ ($0 < a \leq \infty$) centered at the origin. Furthermore, we suppose that $\Psi (x) = \phi(r)$, where $r=|x|=\sqrt{x_1^2+ \cdots + x_n^2}$ and $\phi$ is a convex function.
 
 We will write $f(r) = \frac{1}{r} \phi'(r)$, so $f\in C^{\infty }[0, a )$ and
$$ \Psi_{ij} = f \delta_{ij} + \frac{f'}{r}x_ix_j .$$
For convenience, we write $h = f + rf'$. It is straightforward to see that $g_{\Psi}$ is a metric if and only if both $f$ and $h$ are positive on $[0, a)$, as the eigenvalues of the matrix $(\Psi_{ij})$ are $h$ and $(n-1)$ copies of $f$.

 For such metrics, we provide an alternate characterization of completeness.

\begin{lemma} \label{completeness}
A rotationally symmetric Hessian metric $g_{\Psi}$  is complete if and only if $$ \int_{0}^a \sqrt{h} \,dr = \infty .$$
\end{lemma}
\begin{proof}
Assume that $g_\Psi$ is complete. Consider $\gamma (r) = (r, 0 , \ldots , 0)$, $0\leq r<a$. We have $\gamma'(r) =\frac{\partial}{\partial x_1}$, so $|\gamma'(r)|^2_{g} = \Psi_{11} = f+ rf'=h$. Therefore the length of $\gamma$ is equal to $\int_{0}^a\sqrt{h} \,dr$, which must be infinite for the metric to be complete.

Conversely, assume that this integral is infinite. Let $\gamma (s)=(x_1(s), \ldots , x_n(s))$, $0\leq s<b$ (where $0<b\leq \infty$), be a smooth curve in $B_a$ approaching the boundary. We may assume that $s$ is the arc-length parameter in the Euclidean metric $g_0$, namely, $x_1'^2 + \cdots + x_n'^2 =1$. We have
$$ |\gamma'(s)|^2_g = \sum_{i,j=1}^n x_i' x_j' \Psi_{ij} = f\sum_{i=1}^2 x_i'^2 + \frac{f'}{r} \left(\sum_{i=1}^n x_i x_i'\right)^2 = f + \frac{f'}{r} (rr')^2. $$
By the Cauchy-Schwartz inequality, we have
$$ r^2r'^2 = \left(\sum x_ix_i'\right)^2 \leq \sum x_i^2 \cdot \sum x_i'^2 = r^2\cdot 1,$$
so $r'^2\leq 1$, and
$$|\gamma '|^2_g =f + f'rr'^2 \geq f r'^2 + f'r r'^2 = h r'^2, $$
thus $|\gamma '|_g \geq \sqrt{h} \,r'$, and we have
$$ \int_0^b |\gamma'(s)|_g ds \geq \int_0^b \sqrt{h} r' ds = \int_0^a \sqrt{h} \,dr = \infty. $$
This proves that $g_{\Psi}$ is complete.
\end{proof}

\section{$O(n)$-symmetric K\"ahler Sasaki metrics with non-negative bisectional curvature} 
\label{Bisectional}

In this section, we consider $O(n)$-symmetric K\"ahler metrics with non-negative bisectional curvature. For these metrics, we can simplify the curvature formulas by computing several further derivatives of the rotationally symmetric potential function $\Psi (x) = \phi(r)$. 
In particular, we find that
\begin{equation}
 \Psi_i =\frac{P'}{r}x_i = fx_i, \ \ \ \ \Psi_{ij} = f\delta_{ij} + \frac{f'}{r}x_ix_j, \ \ \ \ \Psi^{ij} = \frac{1}{f} \left( \delta_{ij} - \frac{f'}{rh} x_ix_j \right) .
 \end{equation}
For the third and fourth derivatives, we have
\begin{eqnarray}
 \Psi_{ijk} & = & \dot{f} \left( \delta_{ij} x_k + \delta_{jk} x_i  + \delta_{ik} x_j  \right)  + \ddot{f} \,x_ix_jx_k  \\
 \Psi_{ijk\ell } & = & \dot{f} \left( \delta_{ij}\delta_{k\ell } + \delta_{jk}\delta_{i\ell }  + \delta_{ik}\delta_{j\ell } \right) + \dddot{f} \,x_ix_jx_k x_{\ell} \,+ \nonumber \\
 & &  + \,\ddot{f} \left( \delta_{ij} x_k x_{\ell } + \delta_{ik} x_j x_{\ell }  + \delta_{i\ell} x_j x_{k } + \delta_{jk} x_i x_{\ell } + \delta_{j\ell} x_i x_{k} + \delta_{k\ell} x_i x_j \right)
\end{eqnarray}
where we have used the notation
$$ \dot{f}=\frac{1}{r}f', \ \ \ \ddot{f} = \frac{1}{r}\left( \frac{1}{r}f' \right)', \ \ \ \dddot{f} = \frac{1}{r} \left(\frac{1}{r} \left(\frac{1}{r}f' \right)' \right)' .$$

Now suppose that $u=(u_1, \ldots , u_n)$ and $v=(v_1, \ldots , v_n)$ are unit vectors in ${\mathbb C}^n$ (in Euclidean norm). We denote
$$ \alpha_u = \left \langle u, \frac{x}{r} \right \rangle = \sum \frac{x_iu_i}{r}, \ \ \ \alpha_v = \left \langle v, \frac{x}{r} \right \rangle, \ \ \ \beta =  \left \langle u,v \right \rangle, \ \ \ \lambda = \left \langle u,\overline{v} \right \rangle.  $$

Using this notation, we find that
\begin{eqnarray}
 \Psi_{uvp} & = & \dot{f} \left( \beta x_p + r\alpha_uv_p + r\alpha_v u_p   \right)  + r^2 \ddot{f} \,\alpha_u\alpha_vx_p  \\
 \Psi_{uvx} & = & \sum_{p=1}^n \Psi_{uvp}\,x_p \ = \ r^2\dot{f} \left( \beta  + 2\alpha_u \alpha_v \right)  + r^4 \ddot{f} \,\alpha_u\alpha_v  \\
 \Psi_{u\overline{u}v\overline{v}} & = & \dot{f} \left( 1+|\beta |^2+ |\lambda |^2)  \right) + r^2 \ddot{f} \left( |\alpha_u|^2 +|\alpha_v|^2  +2Re \{ \overline{\beta} \alpha_u \alpha_v + \lambda \overline{\alpha_u}\alpha_v \} \right)  \nonumber \\ & & + \,r^4 \dddot{f} |\alpha_u|^2 |\alpha_v|^2. \label{eq:4thderi}
\end{eqnarray}
Furthermore, it follows that
\begin{eqnarray*}
\sum_p |\Psi_{uvp}|^2 & = & r^2\dot{f}^2 \left( 
|\beta |^2 +  |\alpha_u|^2 + |\alpha_v |^2 + 4 \mathcal{B}  +2\mathcal{C} \right)  \\
& & +  r^6 \ddot{f}^2 |\alpha_u\alpha_v|^2 +  r^4  \dot{f}\ddot{f} \left( 4|\alpha_u\alpha_v|^2 + 2\mathcal{B}  \right) \textrm{ and } \end{eqnarray*}

\begin{eqnarray*} |\Psi_{uvx}|^2 &  =  & r^4\dot{f}^2 \left( |\beta|^2  + 4|\alpha_u \alpha_v|^2 + 4 \mathcal{B}  \right) \\
 & & + \, r^8 \ddot{f}^2  |\alpha_u\alpha_v|^2 +  r^6\dot{f} \ddot{f} \left( 4|\alpha_u\alpha_v|^2 + 2 \mathcal{B} \right) 
\end{eqnarray*}
where $\mathcal{B} =  Re (\overline{\beta}\alpha_u\alpha_v )$, and $\mathcal{C} =Re (\lambda \overline{\alpha_u}\alpha_v)$. Using the fact that $h-r^2\dot{f}=f$,  we get
\begin{eqnarray} \label{thirdderivterm}
\sum_{p,q} \Psi_{uvp}\, \overline{\Psi_{uvq}} \,\Psi^{pq} & = &  \frac{1}{f} \sum_p |\Psi_{uvp}|^2 - \frac{\dot{f}}{fh} |\Psi_{uvx}|^2 \nonumber \\
 & = & \frac{r^2 \dot{f}^2 }{f} \left( |\alpha_u|^2 + |\alpha_v|^2+ 2\mathcal{C} \right) + \frac{r^2 \dot{f}^2} {h} |\beta|^2  + \\
 & & + \, |\alpha_u\alpha_v|^2 \left( \frac{ 4r^4\dot{f} \ddot{f} + r^6 \ddot{f}^2 }{h} - \frac{4r^4 \dot{f}^3 } {fh} \right) +  \mathcal{B} \, \frac{2 r^4 \dot{f} \ddot{f} + 4r^2 \dot{f}^2 }{h}. \nonumber
 \end{eqnarray}
Combining (\ref{eq:4thderi})  with (\ref{thirdderivterm}), we obtain
\begin{eqnarray}
4 R^{(h)}_{u\overline{u}v\overline{v}} & = & - \Psi_{u\overline{u}v\overline{v}} + \Psi_{uvp}\, \overline{\Psi_{uvq}} \,\Psi^{pq} \nonumber \\
& = & - \dot{f}(1+|\lambda|^2) -\frac{f\dot{f}} {h}  |\beta|^2 + \left( \frac{r^2 \dot{f}^2 }{f} -r^2\ddot{f} \right) \,  \left( |\alpha_u|^2 + |\alpha_v|^2+ 2\mathcal{C} \right)
 + \nonumber \\
 & & + \, |\alpha_u\alpha_v|^2 \, \left(  \frac{4r^4 \dot{f} \ddot{f} +r^6 \ddot{f}^2 } {h} - \frac{ 4r^4 \dot{f}^3 }{fh}     - r^4 \dddot{f} \right) +  \mathcal{B} \, \frac{4 r^2 \dot{f}^2 - 2r^2 f\ddot{f}}{h}. \label{eq:curvature}
 \end{eqnarray}

In the special case when $u=\frac{x}{r}$ and $v\perp u$, we have $\beta =\lambda =\alpha_v=0$, $\alpha_u=1$, and $\mathcal{C} =\mathcal{B} =0$, so in this case the curvature becomes
\begin{equation}
 -\dot{f} + \frac{r^2\dot{f}^2}{f} - r^2 \ddot{f} = -f (\log f)''. \label{eq:A}
\end{equation}
Note that in this case we have $\Psi_{u\overline{v}} =0$, so $ -f (\log f)''$ is a value of the orthogonal bisectional curvature of the K\"ahler Sasaki metric $h_{\Psi}$.

\begin{proposition}
Let $\Psi$ be a strongly convex rotationally symmetry function on the ball $B_a\subseteq {\mathbb R}^n$ of radius $0<a\leq \infty$. Assume that the K\"ahler Sasaki metric $\mathfrak{h}_{\Psi}$ is complete. If either $n= 2$ and the bisectional curvature is everywhere nonnegative, or if $n\geq 3$ and the orthogonal bisectional curvature is everywhere nonnegative, then $a=\infty $ and the metric $\mathfrak{h}_{\Psi}$ is the flat Euclidean metric on ${\mathbb C}^n$.
\end{proposition}

\begin{proof} Suppose that $\mathfrak{h}_{\Psi}$ were to have everywhere nonnegative orthogonal bisectional curvature ($n\geq 3$) or nonnegative bisectional curvature ($n=2$). If  $\alpha_u=\alpha_v=0$, then $\mathcal{C}=\mathcal{B}=0$, and
$$ 4 R^{(\mathfrak{h})}_{u\overline{u}v\overline{v}} = - \dot{f} \left( 1+ |\lambda|^2 + \frac{f}{h}|\beta|^2 \right) .$$
Note that when $n\geq 3$, we can choose such $u$ and $v$ so that $\Psi_{u\overline{v}}=0$. So under the assumptions on the curvature, we always have $\dot{f} \leq 0$ and $(\log f)'' \leq 0$.
For convenience, denote $F=\log f$, which is a smooth function on $[0,a)$. We have $F'\leq 0$, $F''\leq 0$, and
$$ h= (rf)' = (1+rF')e^F > 0. $$
If $a< \infty $, then since $h\leq e^F \leq e^F(0)$ as $F$ is non-increasing, we see that $\int_0^a \sqrt{h} \, dr < \infty$, contradicting with the completeness of the metric. So we may assume that $a=\infty$. Since $h>0$, we find that
$$ 0\geq F' > - \frac{1}{r} , \ \ \ 0<r<\infty. $$
This forces $\lim_{r\rightarrow \infty }F(r)=0$. But $F'$ is nonpositive and non-increasing, so must be constantly zero. This implies that $f$ is a positive constant, hence $\mathfrak{h}_{\Psi}$ the ($f$-multiple of the standard) flat complex Euclidean metric on ${\mathbb C}^n$.
\end{proof}

 It is worth noting that Yau's Uniformization Conjecture (partially proven by Liu \cite{Liu}) states that any complete K\"ahler metric with non-negative bisectional curvature is biholomorphic to $\mathbb{C}^n$. We have made strong assumptions on the structure of our manifold (in particular $O(n) \times \mathbb{R}^n$-symmetry), which is why it is possible to conclude that the metric is actually flat, and not simply that it is biholomorphic to $\mathbb{C}^n.$

\section{The Anti-Bisectional Curvature}
\label{antibisectionalsection}

For K\"ahler Sasaki metrics, there is also a notion of ``anti-bisectional curvature," which is defined as follows:
\begin{equation}
\mathfrak{A} (u,v) = \sum_{i,j,k,\ell} \left( -\Psi_{ijk\ell} + \sum_{p,q} \Psi_{ijp} \Psi_{k\ell q} \Psi^{pq} \right)u_iu_j v_k v_{\ell}, 
\end{equation} 
where $u=(u_1, \ldots , u_n)$, $v=(v_1, \ldots , v_n)$ are vectors in ${\mathbb R}^n$. 

\begin{definition}[Non-negative anti-bisectional curvature] \label{antibisectional}
The K\"ahler metric $\mathfrak{h}_\Psi$ has non-negative orthogonal anti-bisectional curvature (abbreviated (NAB)), if $\mathfrak{A} (u,v)\geq 0$ for any $u$, $v\in {\mathbb R}^n$. 
\end{definition}

\begin{definition}[Non-negative orthogonal anti-bisectional curvature] \label{NOAB}
The K\"ahler metric $\mathfrak{h}_\Psi$ has non-negative orthogonal anti-bisectional curvature (abbreviated (NOAB)), if $\mathfrak{A} (u,v)\geq 0$ for any $u$, $v\in {\mathbb R}^n$ such that $\Psi_{uv} = \sum_{i,j} u_iv_j \Psi_{ij} =0$. 
\end{definition}
Note that if we write $U= \sum_i u_i \frac{\partial }{\partial z_i}$ and $V= \sum_i v_i \frac{\partial }{\partial z_i}$. Then we have
\begin{equation}
\mathfrak{A} (u,v) = 4 R^{(\mathfrak{h})} (U, \overline{V}, U, \overline{V}).
\end{equation}
 So (NOAB) and (NAB) are positivity conditions on the curvature of $\mathfrak{h}$, but are different from requiring $\mathfrak{h}$ to have nonnegative orthogonal bisectional curvature or non-negative bisectional curvature. 
If one requires a K\"ahler metric  $\omega$ to satisfy the curvature condition $R(X,\overline{Y}, X, \overline{Y})\geq 0$ for any type $(1,0)$ vectors $X$ and $Y$, then it is necessary for $\omega$ to have constant holomorphic sectional curvature and so be a complex space form. As such, the above (NOAB) and (NAB) conditions are specialized to K\"ahler Sasaki metrics and do not generalize naively to arbitrary K\"ahler metrics. We leave it as an open question to find meaningful generalization of these conditions for more general K\"ahler metrics.

\subsection{The orthogonal anti-bisectional curvature of $O(n)$-symmetric metric}

We now return to the case where the tube domain is rotationally symmetric. We want understand when $T B_a$ will have (NOAB), which means
\begin{equation} 
\mathfrak{A} (u,v) = - \Psi_{uuvv} + \sum_{p,q} \Psi_{uup} \Psi_{vvq} \Psi^{pq} \geq 0 
\end{equation}
for any $u$, $v\in {\mathbb R}^n$  such that $\Psi_{uv}=0$. To find conditions which ensure (NOAB), we will rewrite the above expression in terms of some auxilliary functions.

Without loss of generality, we may assume that $|u|=|v|=1$ under the Euclidean norm. Let us write $\beta = \langle u,v\rangle = \sum_i u_iv_i$, $\alpha_u = \langle u, \frac{x}{r}\rangle$, and $\alpha_v$ similarly. Then since
$$ 0= \Psi_{uv} = \Psi_{ij} u_i v_j = f \beta + r^2\dot{f}\alpha_u \alpha_v ,$$
we find
\begin{equation}
 \beta = -\frac{r^2 \dot{f}}{f} \alpha_u\alpha_v = \left(1-\frac{h}{f} \right) \alpha_u\alpha_v.
\end{equation} 
From the formulas for $\Psi^{ij}$, $\Psi_{ijk}$, and $\Psi_{ijk\ell}$, we get
\begin{eqnarray} \label{3deriv1}
\Psi_{uuk} & = &  \dot{f} (x_k + 2r\alpha_u u_k) + r^2\ddot{f} \alpha_u^2 x_k \\
\Psi_{uux} & = &  r^2\dot{f} (1 + 2\alpha_u^2) + r^4\ddot{f} \alpha_u^2 \label{3deriv2} \\
\Psi_{uuvv} & = &  \dot{f} (1 + 2\beta^2) + r^2\ddot{f} (\alpha_u^2 + \alpha_v^2 + 4\beta \alpha_u \alpha_v) + r^4 \dddot{f} \alpha_u^2 \alpha_v^2. \label{4deriv}
\end{eqnarray}
Combining Equations (\ref{3deriv1}) and (\ref{3deriv2}), we find
\begin{eqnarray}
\Psi_{uuk} \Psi_{vvk} & = & 
    r^2\dot{f}^2 ( 1+ 2\alpha_u^2+2\alpha_v^2+4\beta \alpha_u\alpha_v) \nonumber \\
    & & +  r^4 \dot{f}\ddot{f} (\alpha_u^2 + \alpha_v^2 + 4\alpha_u^2\alpha_v^2)  + r^6 \ddot{f}^2 \alpha_u^2\alpha_v^2
 \label{3deriv3} \\
\Psi_{uux} \Psi_{vvx} & = &  r^4\dot{f}^2 ( 1+ 2\alpha_u^2+2\alpha_v^2+4\alpha_u^2\alpha_v^2) \nonumber \\
&  & + r^6 \dot{f}\ddot{f} (\alpha_u^2 + \alpha_v^2 + 4\alpha_u^2\alpha_v^2)  + r^8 \ddot{f}^2 \alpha_u^2\alpha_v^2, \label{3deriv4} 
\end{eqnarray}
where $k$ is summed in the first line. Since $\beta = (1-\frac{h}{f})\alpha_u\alpha_v$ and 
$$\frac{1}{f} - \frac{\dot{f}}{fh}r^2 = \frac{1}{f} - \frac{h-f}{fh} = \frac{1}{h}, $$
we obtain from Equations (\ref{3deriv3}) and (\ref{3deriv4}) that
\begin{eqnarray*}
 \sum_{p,q}\Psi_{uup}\Psi_{vvq} \Psi^{pq} & = &   \frac{1}{f} \sum_k \Psi_{uuk} \Psi_{vvk} - \frac{\dot{f}}{fh} \Psi_{uux}\Psi_{vvx} \\
& = & \frac{1}{f} \dot{f}^2 \left(-4\frac{h}{f} \right)\alpha_u^2\alpha_v^2 + \frac{1}{h} r^2\dot{f}^2 ( 1+ 2\alpha_u^2+2\alpha_v^2+4\alpha_u^2\alpha_v^2) + \\
& &  + \  \frac{1}{h}  r^4 \dot{f}\ddot{f} (\alpha_u^2 + \alpha_v^2 + 4\alpha_u^2\alpha_v^2)  + \frac{1}{h} r^6 \ddot{f}^2 \alpha_u^2\alpha_v^2.
\end{eqnarray*}
Using the above identity and Equation (\ref{4deriv}), we find that
\begin{equation}
\mathfrak{A} (u,v) = A + B (\alpha_u^2 + \alpha_v^2) + C \alpha_u^2 \alpha_v^2,
\end{equation}
where
\begin{eqnarray*}
A  & = & \frac{1}{h}r^2\dot{f}^2 - \dot{f} , \\
B & = & \frac{2}{h} r^2\dot{f}^2 + \frac{1}{h} r^4\dot{f} \ddot{f} - r^2 \ddot{f}, \\
C& = &  - \frac{4h}{f^2}r^2\dot{f}^2 + \frac{4}{h}r^2\dot{f}^2  + \frac{4}{h}r^4 \dot{f}\ddot{f} + \frac{1}{h}r^6\ddot{f}^2- \ 2(1-\frac{h}{f})^2 \dot{f} - 4(1-\frac{h}{f})r^2 \ddot{f} - r^4 \dddot{f}.
\end{eqnarray*}
Using the relation $\ r^2\dot{f} = h-f$ to simplify $A, B$ and $C$, we find
\begin{eqnarray} \label{Ayy}
A  & = &  -\frac{f}{h} \dot{f},  \\
B & = &   \frac{r^2}{h}(2\dot{f}^2 -f\ddot{f}), \label{Bee} \\
C& = & \left(-4\frac{f}{h} + 2 + 8\frac{h}{f} -6 \frac{h^2}{f^2} \right) \dot{f} + 4\left(\frac{h}{f} - \frac{f}{h} \right) r^2\ddot{f} + \frac{r^6}{h}\ddot{f}^2 -r^4\dddot{f}. \label{Cee}
\end{eqnarray}
Now let us try to understand the (NOAB) condition for these metrics. Recall that this says that $\mathfrak{A} (u,v)\geq 0$ for any unit vectors $u,v\in {\mathbb R}^n$ such that $\Psi_{uv}=0$ (or equivalently, vectors for which $\beta = (1-\frac{h}{f})\alpha_u\alpha_v$). Fixing the unit vector $\frac{x}{r}$ and denoting its orthogonal complement in ${\mathbb R}^n$ by $W$, we write 
$$ u = \alpha_u \frac{x}{r} + u', \ \ \ v = \alpha_v \frac{x}{r} + v', $$
where $u'$, $v'\in W$. Since $u$, $v$ are unit vectors, we have $|u'|^2=1-\alpha_u^2$ and $|v'|^2=1-\alpha_v^2$. Also, 
$$ \beta = \langle u,v\rangle = \alpha_u\alpha_v + \langle u', v'\rangle ,$$ 
so we get
\begin{equation*}
-\frac{h}{f} \alpha_u\alpha_v = \langle u', v'\rangle  = |u'| \, |v'| \cos \theta 
\end{equation*}
where $\theta$ is the angle between $u'$ and $v'$. When $n=2$, $W$ is one dimensional, hence we get
\begin{equation}
\frac{h^2}{f^2} \alpha_u^2 \alpha_v^2 = (1-\alpha_u^2) (1-\alpha_v^2) 
\end{equation}
When $n\geq 3$, we find instead that
\begin{equation}
\frac{h^2}{f^2} \alpha_u^2 \alpha_v^2 \leq (1-\alpha_u^2) (1-\alpha_v^2)
\end{equation}
Writing $s=\alpha_u^2$ and $t=\alpha_v^2$, we know that (NOAB) simply means that the function $F(s,t)= A+B(s+t)+Cst$ is nonnegative for all $s,t\in [0,1]$ such that $\frac{h^2}{f^2} st = (1-s)(1-t)$ when $n=2$ or $\frac{h^2}{f^2} st \leq (1-s)(1-t)$ when $n\geq 3$. Let $\gamma$ be the curve segment:
$$ \frac{h^2}{f^2} st = (1-s)(1-t) , \ \ \  0\leq s,t\leq 1 $$
in the $st$-plane. If $h=f$, then $\gamma$ is the line segment $s+t=1$. If $h\neq f$, then $\gamma$ is the segment of the hyperbola contained in the unit square, and it goes from $(1,0)$ to $(0,1)$. The intersection of $\gamma$ with the diagonal line is $(s_0,s_0)$, where $s_0= \left(1+\frac{h}{f} \right)^{-1}$. Let us define
\begin{equation} \label{Dee}
D = \left( 1+\frac{h}{f} \right)^2 A + 2 \left( 1+\frac{h}{f} \right)B + C,
\end{equation}
where $A$, $B$, and $C$ are given by Equations \ref{Ayy}-\ref{Cee}.
We have the following.
\begin{proposition} \label{NOABforOnsymmetric}
For a rotationally symmetric convex function $\Psi$ on the ball $B_a\subseteq {\mathbb R}^n$, $(T B_a, \mathfrak{h}_\Psi)$ has (NOAB) if and only if 
\begin{enumerate}
\item when $n=2$:  $A+B$, $D$ are everywhere non-negative 
\item when $n\geq 3$:  $A$, $A+B$, $D$ are everywhere non-negative.
\end{enumerate}
Here, A, B and D are given by \ref{Ayy}, \ref{Bee} and \ref{Dee}, respectively.
\end{proposition}

\begin{proof}
As noted above, (NOAB) means that $F(s,t)=A+B(s+t)+Cst\geq 0$ for any $0\leq s,t\leq 1$ satisfying the above constraint condition, which means that $(s,t)$ lies on the curve segment $\gamma$ if $n=2$ and $(s,t)$ lies in the sub-region $\Omega$ in the unit square bounded by the coordinate axes and $\gamma$. 

First let us assume that $n=2$. Since $F(1,0)=A+B$ and $F(s_0,s_0)=s_0^2D$, we know that both $A+B$ and $D$ are non-negative when (NOAB) is satisfied. Conversely, if both $A+B$ and $D$ are non-negative, then it is easy to see that $F$ is non-negative along the entire curve segment $\gamma$, hence (NOAB) is satisfied. 
 
Now assume that $n\geq 3$. If (NOAB) is satisfied, then $A=F(0,0)$, $A+B=F(1,0)$, $D=F(s_0,s_0)$ are all non-negative. Conversely, suppose that $A$, $A+B$, $D$ are all non-negative. Then it is not hard to see that  $F$ is non-negative along the boundary of the domain $\Omega$. We claim that $F$ is also non-negative in the interior of $\Omega$. To see this, suppose that $(s_1, t_1)$ is a critical point of $F$ in the interior of $\Omega$. Then we have $B+Cs_1=B+Ct_1=0$ and $s_1,t_1\in (0,1)$. So $B$ and $C$ are non-zero and with opposite sign.

We consider two separate cases. \begin{enumerate}
    \item If $C>0$, then $s_1=t_1=-\frac{B}{C}$ so $0<-B<C$. On the other hand, since  $A+B\geq 0$, we have $-B \leq A$. So $B^2\leq AC$, hence
$$ F(s_1, t_1) = A - \frac{B^2}{C} = \frac{1}{C} ( AC-B^2) \geq 0. $$
\item On the other hand, if $C<0$, then $AC-B^2\leq AC \leq 0$, hence $F(s_1, t_1)=\frac{1}{C} ( AC-B^2) \geq 0$.
\end{enumerate} 
In either case, $F(s_1,t_1) \geq 0$, so the orthogonal anti-bisectional curvature is non-negative. This completes the proof of the proposition. 
\end{proof}

The above proposition says that to find $O(n)$-symmetric metrics with (NOAB), we need only to focus on the derived functions $A$, $A+B$, and $D$. As such, we will analyze these terms more carefully. First let us try to understand the second order derivative term $A+B$. To make the equation more tractable, it is helpful to define further auxilliary functions. To this end, we define
\begin{equation} \label{ellandlambda}
\ell = \frac{1}{f}, \ \ \ \lambda =\ell - r\ell'
\end{equation}
Recall that in terms of the original convex function $\phi:\mathbb{R}_{\geq 0} \to \mathbb{R}$ used to define a radially symmetric potential, $\ell=\frac{r}{\phi^\prime(r)}$.

It follows immediately that $\ell$ is a positive function. From the formula $h=(rf)'= \frac{\ell - r\ell'}{\ell^2}=\frac{\lambda }{\ell^2}$, we know that $\lambda $ is also a positive function. By a straight forward computation, we find 
\begin{eqnarray}
A & = & \frac{\ell'}{r \lambda \ell}\\
A+B & = & \frac{\ell''}{\lambda \ell}\\
D & = & \left(\!-\frac{1}{r\lambda \ell} + \frac{1}{r\ell^2} -\frac{3\lambda }{r\ell^3} + \frac{6\lambda^2}{r\ell^4} \right) \ell' + \left(\!-\frac{7}{\lambda \ell^2} + \frac{6}{\ell^3} + \frac{8\lambda}{\ell^4} \right) \ell'^2  \nonumber \\
 & & + \ \left( \frac{4r}{\lambda \ell^3} + \frac{6r}{\ell^4} \right) \ell'^3  +  \frac{4r^2}{\lambda \ell^4} \ell'^4 + \left( \frac{6}{\lambda \ell} - \frac{1}{\ell^2} -\frac{4\lambda }{\ell^3} \right) \ell'' \nonumber \\
 & & - \, \left( \frac{2r}{\lambda \ell^2} +\frac{6r}{\ell^3} \right) \ell' \ell''  - \frac{4r^2}{\lambda \ell^3} \ell'^2\ell'' + \frac{r^2}{\lambda \ell^2}\ell''^2 + \frac{r}{\ell^2} \ell'''.
\end{eqnarray}
From this, we get the following:

\begin{proposition}
Let $\Psi$ be a strictly convex rotationally symmetric function on the ball $B_a\subseteq {\mathbb R}^n$ (where $0<a\leq \infty$) so that the K\"ahler Sasaki metric $g_{\Psi}$ is complete and satisfies (NOAB). If $n\geq 3$, then $a=\infty$. 
\end{proposition}

\begin{proof}
By the previous proposition, we have $A$, $A+B$, $D$ all non-negative. Thus $\ell'\geq 0$, $\ell''\geq 0$. So $h=\frac{\lambda}{\ell^2}$ is non-increasing since $\lambda' = -r \ell'' \leq 0$. The completeness of $g_{\Psi}$ means that the integral $\int_0^{a} \sqrt{h}dr = \infty$. So $a$ must be $\infty$. 
\end{proof}

In particular, any such metric must be biholomorphic to $\mathbb{C}^n$ with its standard complex structure. This analysis also allows us to find several examples of such metrics, two of which we provide here.

\begin{example} \label{Example1}
Consider the function $\ell = C+r$ on $[0,\infty )$, where $C>0$ is a constant. We have $\ell'=1$, $\ell''=0$, and $\lambda = C$. Hence
\begin{eqnarray*}
r\ell^4 D &= & \frac{1}{C} \left(4r^3+4r^2\ell  -7r\ell^2 -\ell^3 \right) + \left( 6r^2 +6r\ell +\ell^2 \right)  + \ C \left(8r-3\ell  \right) + 6C^2 \\
& = & - \frac{1}{C}\left( 13Cr^2 +10C^2r +C^3 \right) + \left( 13r^2 +8Cr+C^2 \right)  + C(5r-3C) + 6C^2 \\
& = & 3Cr + 3C^2 \ = \ 3C \,\ell.
\end{eqnarray*}
So the metric has (NOAB). Also, since $h=\frac{C}{\ell^2}$, we have $\int_0^{\infty} \sqrt{h}dr = \infty$, hence $g_{\Psi}$ is complete.
\end{example}

It is worth noting that the curvature of this example is unbounded near the origin, although that is not important for the analysis in this paper.

\begin{example} \label{ComplicatedExample}
Consider the function
$$ \ell = r + \frac{1}{L^2}, \ \ \ L = \log (c+r) $$
for $c \log(c)^3 \geq 2$. In this case, if we write $R=c+r$, then we have that
$$ \lambda = \frac{1}{L^2} + \frac{2r}{RL^3}>0, \ \ \ \ell' = 1-\frac{2}{RL^3} > 0, \ \ \ \ell'' = \frac{2}{R^2L^3} + \frac{6}{R^2L^4} > 0, $$
and $\sqrt{h} = \frac{\sqrt{\lambda}}{\ell} \sim \frac{1}{RL}$, so $\int_0^{\infty } \sqrt{h}dr = \infty$ and the metric is complete.
Furthermore, $D\geq 0$ so this is another examples of an $O(n)$-symmetric complete metric with (NOAB).
\end{example}

The verification of this example is more involved and we postpone a proof that $D \geq 0$ to the Appendix.

\section{The MTW tensor and the Regularity of Optimal Transport}
\label{Regularityoftransport}

Apart from complex geometry, the primary motivation for considering anti-bisectional curvature arises from optimal transport. To explain this, we first discuss some preliminary background on optimal transport.  For a more complete reference on this topic, we refer the reader to the survey paper of DePhilippis and Figalli \cite{DPF} or the book by Villani \cite{VillaniOTON}.

The original transport problem was considered by Monge in 1781 \cite{Monge}. In his work, he sought to find the most cost-efficient way to transport rubble (\textit{d\'eblais}) into a desired configuration to build a fortification (\textit{remblais}). In the modern setting, this problem is formalized in terms of the Kantorovich formulation.

Given probability spaces $(X, \mu)$ and $(Y,\nu)$ and a lower semi-continuous cost function $c(x,y):X \times Y \to \mathbb{R}$, the Kantorovich problem seeks to find a coupling $\pi$ of $\mu$ and $\nu$ which achieves
\begin{equation} \label{Kantorovich}
   \inf_{\pi \in \Pi(\mu,\nu)} \int_{X \times Y} c(x,y)  d \pi.
\end{equation}

Here, $\Pi(\mu,\nu)$ is the set of all couplings of $(X, \mu)$ and $(Y,\nu)$ (i.e. probability measures on $X \times Y$ whose marginals are $\mu$ and $\nu$, respectively). Under mild regularity assumptions on $\mu, \nu$ and $c$, such a coupling exists, which intuitively describes how mass from $X$ is transported to $Y$. In general, the optimal coupling may split mass at a single point and distribution it throughout $Y$. However,
when $X$ and $Y$ are domains in Euclidean space (or domains in a smooth manifold) and certain technical conditions hold, Gangbo and McCann \cite{GangboMcCann} showed that the optimal coupling is induced by a map $T:X \to Y$. More precisely, they showed the following result, which is based off an earlier work of Brenier \cite{Brenier} for the cost function $c(x,y)= |x-y|^2$.

\begin{theorem*}\label{Deterministictransport} 
Let $X$ and $Y$ be two open domains of $\mathbb{R}^n$ and consider a cost function $c:X \times Y \to \mathbb{R}$. Suppose that $ d \mu$ is a smooth probability density supported on $X$ and that $d \nu $ is a smooth probability density supported on $Y$. Suppose that the following conditions hold:

\begin{enumerate}
\item The cost function $c$ is of class $C^4$ with $\| c \|_{C^4(X \times Y)} < \infty$
\item The following two conditions hold, which are collectively called (Twist):
\begin{enumerate}
    \item For any $x \in X$, the map $ Y \ni y \to   D_x c(x,y) \in \mathbb{R}^n$ is injective.
\item For any $y \in Y$, the map $ X \ni x \to  D_y c (x,y) \in \mathbb{R}^n$ is injective.
\end{enumerate}
\item The mixed Hessian matrix $ c_{i,j} = \frac{\partial^2}{\partial x^i \partial y^j} c(x,y) $ is invertible for all $(x,y) \in X \times Y$. In other words, $\det (c_{i,j})(x,y) \neq 0$. Through the rest of the paper, we will denote this condition as (NonDeg).
\end{enumerate}

Then:
\begin{enumerate}
    \item There exists a unique solution to the Kantorovich problem (\ref{Kantorovich}).
    \item This solution is induced by a measurable map $T:X \to Y$ satisfying $T_\sharp \mu = \nu$, which is injective $d \mu$-a.e.
    \item There exists a function $u: X \to \mathbb{R}$ such that $T_u(x) := c\textrm{-}\exp_x(\nabla u(x))$, where $c\textrm{-}\exp$ is the so called $c$-exponential map.
    \item The potential $u$ satisfies the following Monge-Ampere type equation
\begin{equation} \label{Monge Ampere}
 | \det(\nabla T_u(x))| = \frac{ d \mu (x)}{ d \nu(T_u(x))} \hspace{.2in} d \mu-a.e. 
\end{equation}
\end{enumerate}
\end{theorem*}

In other words, the optimal transport is solved by a \textit{transport map}, which sends each point in $X$ to a unique point in $Y$. Furthermore, we can solve for this map by solving a fully non-linear equation of Monge-Amp\`ere type. In this case, it is of interest to determine the continuity properties of $T$, which is to ask whether nearby points in $X$ are sent to nearby points in $Y$. Even for smooth costs and measures, a priori this potential is merely Lipschitz\footnote{A deep theorem of De Phillipis and Figalli \cite{DPF2} shows that for smooth costs and measures, the transport is smooth away from a singular set of measure zero.}, and so the transport map may be discontinuous. Determining the regularity of $T$ is known as the \textit{regularity problem of optimal transport}, and is an active area of research.

In the years after Brenier's initial work, much of the focus for the regularity problem was for the cost function $c(x,y) = |x-y|^2$, in which case the Monge-Ampere equation \ref{Monge Ampere} takes a simple form. For this cost function, Caffarelli and others established a priori interior $C^2$-estimates for weak solutions to (\ref{Monge Ampere}) when
\begin{enumerate}
    \item  $X$ and $Y$ are strictly convex domains, and
    \item the associated densities $d \mu$ and $d \nu$ are bounded away from $0$ and $\infty$.
\end{enumerate}
 These estimates imply a $C^1$-estimate for $T$ (and also higher-order estimates using elliptic bootstrapping for the linearized operator). However, this work did not address the regularity for more general cost functions.

 In 2005, breakthrough work of Ma, Trudinger and Wang found two structural conditions, one on the cost function and another on the domains $X$ and $Y$ that are sufficient\footnote{Loeper \cite{Loeper} showed that these conditions are essentially necessary.} to prove regularity for the optimal transport.
We will discuss the condition on $X$ and $Y$ in Section \ref{SyntheticCurvature}, but
for now, we will focus on the condition on the cost function, which is that the so-called \textit{MTW tensor} is non-negative. 
In order to define the MTW tensor, we first introduce some notation. In the following, we use $c_{I,J}$ to denote $\frac{ \partial^{| I |}}{\partial x^I} \frac{ \partial^{| J |} c}{\partial y^J}$  for multi-indices $I$ and $J$ and $c^{i,j}$ to denote the matrix inverse of the mixed Hessian $c_{i,j}$. Using this notation, for a $C^4$ cost function which satisfies (NonDeg), the MTW tensor (denoted $\mathfrak{S}$) is defined as follows:
\begin{equation} \label{MTWtensor}
 \mathfrak{S} (\xi, \eta)= \sum_{i,j,k,l,p,q,r,s} (c_{ij,p}c^{p,q}c_{q,rs}-c_{ij,rs})c^{r,k}c^{s,l} \xi^i \xi^j \eta^k \eta^l
\end{equation}
In this formula, $\eta$ is a vector and $\xi$ is a covector. The MTW tensor is a fourth-order quantity which scales quadratically in $\xi$ and $\eta$. Although it is not immediately obvious, this expression transforms tensorially under change of coordinates.

\begin{definition*}
A $C^4$ cost function with invertible mixed Hessian satisfies:
\begin{enumerate}
\item (MTW)  If $ \mathfrak{S}(\xi,\eta) \geq 0$ for all vector-covector pairs satisfying $\eta(\xi)=0$. Such cost functions are also said to be \textit{weakly-regular}.
\item (MTW($\kappa)$) If $ \mathfrak{S}(\xi,\eta) \geq \kappa | \xi |^2 | \eta |^2 $ for all vector-covector pairs satisfying $\eta(\xi)=0$. This is also known as strong MTW non-negativity.
\item (NNCC)  If $ \mathfrak{S}(\xi,\eta) \geq 0$ for all vector-covector pairs, not necessarily orthogonal. This condition is also known as non-negative cost curvature.
\end{enumerate}
\end{definition*}

As mentioned previously, to prove regularity for optimal transport, it is necessary to prove a priori estimates for equations of the form (\ref{Monge Ampere}). A full overview of this line of research would take us too far from the main focus of this paper. However, to motivate our considerations, we present one example of such a regularity result, proven by Figalli, Kim, and McCann \cite{FKM2}.

\begin{theorem*}[\cite{FKM2}, Theorem 2.1] \label{Holderestimate}

Let $X$ and $Y$ be two domains in $\mathbb{R}^n$ and let $c$ be a cost function $c:X \times Y \to \mathbb{R}$. Consider two probability densities $f(x)$ and $g(y)$ supported on $X$ and $Y$ and suppose that the following conditions hold:

\begin{enumerate}
\item The cost function $c$ is of class $C^4$ with $\| c \|_{C^4(X \times Y)} < \infty$
\item The cost function satisfies (Twist) and (NonDeg)
\item The density $f$ is bounded from above on $X$ and the density $g$ bounded away from both zero and infinity on $Y$.
\item The domains $X$ and $Y$ are uniformly \textit{relatively $c$-convex}. (See Definition \ref{relcconvex})
\item The cost function satisfies (MTW).
\end{enumerate}
Then the optimal transport from $f$ to $g$ is induced by a map $T \in C^{\alpha} (\overline{X^\prime})$ $X^\prime \subset X$ is an open set with $f$ bounded uniformly away from zero.
\end{theorem*}

\subsection{The Anti-Bisectional Curvature and the MTW tensor.}

To relate the MTW tensor to the curvature of tube domains, we specialize our attention to cost functions of the form $c(x,y) = \Psi(x-y)$ for some strongly convex function $\Psi:\Omega \to \mathbb{R}$ (henceforth \textit{ $\Psi$-costs}). Such cost functions were first studied by Gangbo and McCann \cite{GangboMcCann}, although they did not use this terminology.

For a $C^4$ $\Psi$-cost, the MTW tensor is proportional to the orthogonal anti-bisectional curvature of the associated K\"ahler Sasaki metric (Theorem 6 of \cite{KhanZhang}).  As a result of this, a $\Psi$-cost satisfies (MTW) iff the associated K\"ahler Sasaki metric satisfies (NOAB). Furthermore, the cost-curvature is proportional to the anti-bisectional curvature, so (NNCC) for a $\Psi$-cost corresponds to (NAB) for the associated K\"ahler Sasaki metric.

From this observation, we can use our examples to generate $\Psi$-costs with (MTW). Furthermore, these costs will satisfy a growth condition at infinity, which corresponds to completeness for the K\"ahler Sasaki metric (or equivalently completeness of the underlying Hessian manifold). As a shorthand for this, we say that a $\Psi$-cost is \textit{complete} if the associated Hessian manifold is complete as a Riemannian manifold.

\begin{example}
The cost function $c(x,y)= \|x-y\| - C \log(\|x-y\|+C )$ is a complete cost function which satisfies (MTW).
\end{example}
This is the cost function corresponding to Example \ref{Example1} (after integrating out to solve for $\phi$). Unfortunately, it is not possible to write out a closed form cost associated with Example \ref{ComplicatedExample}, as $r\cdot f$ does not have elementary anti-derivative.

\section{Synthetic Notions of Curvature in Complex Geometry and Optimal Transport} \label{SyntheticCurvature}

In this section, we discuss ways to define synthetic curvature bounds in K\"ahler geometry. We use the term ``synthetic curvature" in the sense of defining curvature for low-regularity metric spaces, which may not be smooth enough for the Riemann curvature tensor to be defined. This notion has also been called ``coarse curvature" (see, e.g., \cite{AcheWarren}), but we will not use this terminology. Although the concept is perhaps best understood by analogy, we will use the following definition for \textit{synthetic curvature bounds}.

\begin{definition*}[Synthetic curvature bounds]
A condition $Q_\kappa$ is a synthetic lower bound for a curvature tensor $S$ if the following two conditions hold:
\begin{enumerate}
\item On a smooth manifold $M$ where $S$ is defined, \[ S \geq \kappa \iff Q_\kappa. \]
\item The condition $Q_\kappa$ is well-defined for spaces with low regularity (where $S$ is not well-defined).
\end{enumerate}
\end{definition*}

One can define synthetic curvature upper bounds analogously. Note that we have purposely left the condition $Q_\kappa$ and the curvature tensor $S$ ambiguous, so as to make this definition as general as possible. Depending on the context, $S$ might be the sectional curvature, Ricci curvature, scalar curvature, or any other sort of curvature. The main goal of this section is to define synthetic versions of (NOAB) and (NAB) on tube domains. However, it is instructive to first consider several examples of synthetic curvature.

 To motivate the definition of a synthetic curvature bound, it is worth considering the CAT($\kappa$)-inequality, which is the prototypical example.
 
 \begin{theorem*} [CAT($\kappa$)-inequality]
 Suppose $M$ is a Riemannian manifold with sectional curvature $S$ satisfying $S \geq \kappa$. Denote the distance function on $M$ by $d$. Let $\triangle pqr$ be a geodesic triangle in $M$ (i.e. a triangle whose sides are geodesics) such that 
 \begin{enumerate}
     \item the sides $\overline{pq}, \overline{pr}$ and $\overline{qr}$ are minimal, and
     \item if $\kappa>0$, all of the sides have length at most $\frac{\pi}{\sqrt{ \kappa}}$.
 \end{enumerate}  
 
 For comparison, let $M_\kappa$ be a simply connected space of constant curvature $\kappa$ and consider $\triangle p'q'r'$ a geodesic triangle in $M_\kappa$ with
 \begin{enumerate}
     \item $\textrm{length}(\overline{pq}) = \textrm{length}(\overline{p'q'})$,
     \item $\textrm{length}(\overline{pr}) = \textrm{length}(\overline{p'r'})$, and
     \item $\textrm{length}(\overline{qr}) = \textrm{length}(\overline{q'r'})$.
 \end{enumerate}
 
 For any pair of points $ (x,y) \in \overline{pq} \times \overline{pr} \subset M \times N$, consider the pair $(x',y') \in \overline{p'q'} \times \overline{p'r'} \subset M_\kappa \times M_\kappa$
  satisfying $d(p,x)= d'(p',x')$ and $d(p,y)= d'(p',y')$.
  
  Then the following inequality holds.
 \begin{equation} \label{narrowtriangle}
     d(x,y) \leq d'(x',y').
 \end{equation}
 \end{theorem*}
 
 In fact, this result \textit{characterizes} sectional curvature bounds, in that whenever a Riemannian manifold has some sectional curvature smaller than $\kappa$, it is possible to find a small geodesic triangle where inequality (\ref{narrowtriangle}) fails. Furthermore, if we use Inequality \ref{narrowtriangle} as the \textit{definition} for sectional curvature bounds, this has the additional advantage in that it is well-defined on spaces which are not smooth manifolds. In this vein, a complete geodesic space which satisfies the inequality \ref{narrowtriangle} is said to be a $CAT(\kappa)$-spaces \cite{Gromov} and play an important role in metric geometry and geometric group theory.  


\subsection{Synthetic Ricci bounds on K\"ahler manifolds}

For a simple though instructive example of this idea in complex geometry, we now discuss a synthetic formulation for Ricci bounds.
On a smooth K\"ahler manifold, the Ricci form is given by the formula
\begin{equation} \label{Ricci}
    \rho = - \sqrt{-1} \partial \bar \partial \log \det \partial \bar \partial \Psi,
\end{equation}
where $\Psi$ is the K\"ahler potential (i.e. The K\"ahler form $\omega$ satisfies $\omega = \partial \bar \partial \Psi$). The Ricci curvature is bounded below (respectively above) by a constant $\kappa$ if
\begin{equation} \label{Riccibound} \rho \geq \kappa \omega \textrm{ (respectively } \leq \kappa \omega \textrm{).} 
\end{equation}

The above inequality should be interpreted in the sense of $(1,1)$-forms. That is to say, given a holomorphic vector $X$, the above inequality implies that $\rho(X, \overline{X}) \geq \kappa \omega( X, \overline{X})$.
To rephrase this in synthetic terms, we consider the function $Q_\kappa = \log \det \partial \bar \partial \Psi + \kappa \Psi$ and say that the Ricci curvature is bounded above (or below) by $\kappa$ whenever
$Q_\kappa$ is plurisubharmonic (plurisuperharmonic)\footnote{This is a slight abuse of notation from our definition of synthetic curvature bounds, where $Q_\kappa$ was a condition, instead of a function, but this is not important.}.
When $\Psi$ is $C^4$, this is equivalent to Ricci bounds in the normal sense. However, this definition does not require $Q_\kappa$ to be $C^2$, so we are able to define Ricci curvature bounds when the potential is only $C^3$ (which is the natural regularity so that the K\"ahler condition $d \omega =0$ is well-defined).

For K\"ahler Sasaki metrics on tube domains, we can simplify this further. For these metrics, the Ricci form simplies to
\[ \rho_{i\bar{j}} = \frac{\partial^2}{\partial x^i \partial x^j} \log \det \left[ \frac{\partial^2 \Psi}{\partial x^k \partial x^l}\right] . \]

Therefore, if we define 
\begin{equation}
    Q_\kappa = \log \det \left[ \frac{\partial^2 \Psi}{\partial x^k \partial x^l}\right] + \kappa \Psi
\end{equation} we can see that the Ricci curvature of a Sasaki metric is bounded below by $\kappa$ if and only if $Q_\kappa$ is convex (on $\Omega$). This immediately implies the following proposition.

\begin{proposition}
A $C^4$ K\"ahler Sasaki metric has Ricci curvature bounded below by $\kappa$ iff the function $Q_\kappa$ satisfies
\[\lambda Q_\kappa(p_1)+ (1-\lambda) Q_\kappa(p_2) \geq Q_\kappa(\lambda p_1 + (1-\lambda)p_2) \]
for all $p_1, p_2 \in \Omega$ and $0< \lambda <1$
\end{proposition}

As such, convexity (or concavity) of $Q_\kappa$ gives a way to \textit{define} bounds on the Ricci curvature. For $C^4$ potentials, this is equivalent to Ricci curvature bounds in the usual sense, but has the advantage of being defined for less smooth potentials.  It is worth noting that there are many other ways of defining synthetic bounds for Ricci curvature, several of which can be defined for less regular metric-measure spaces. We refer to the work of Villani \cite{VillaniSynth} and Ache and Warren \cite{AcheWarren} for some references on this topic.


\subsection{A synthetic (MTW) derived from K\"ahler-Ricci flow}

In recent work, the first and third named authors showed that in two (complex) dimensions, K\"ahler-Ricci flow preserves (NOAB) \cite{KhanZheng}. In other words, given an initial metric which satisfies the condition, all future metrics do as well. Since this flow will instantaneously smooth rough metrics, this motivated the following synthetic version of (MTW).

\begin{definition}[KR weak regularity] \label{KRWeakregularity}
A cost function $c(x,y)=\Psi_0(x-y)$ is KR weakly regular if there exists a solution to the parabolic flow
\begin{equation} \label{Psicostflow}
  \begin{cases}  \frac{ \partial} {\partial t} \Psi(x,t) = 2 \log \left( \det[ \Hess \Psi(x,t)]  \right)\\
 \Psi(x,0) = \Psi_0(x) \end{cases} 
\end{equation} 
 which has (NOAB) for positive time.
\end{definition}

The flow \ref{Psicostflow} induces K\"ahler-Ricci flow on the K\"ahler Sasaki metric, which is the root of the term KR weak regularity. For readers familiar with the regularity theory of optimal transport, note that $\log \left( \det[ \Hess \Psi_0(x)]  \right)$ gives a quantitative measure of how non-degenerate the cost $c(x,y)= \Psi_0(x-y)$ is, so this flow has an interpretation in terms of optimal transport as well as complex geometry. We will not study the existence theory in this paper, but rather use this flow to motivate our study of the regularity of the Monge cost   $c(x,y)=\|x-y\|$.

\subsection{The Monge cost}
\label{Monge cost and KR flow}

At present, there is a well-established theory for the regularity of optimal transport for costs which satisfy the various hypotheses of the Gangbo-McCann and Ma-Trudinger-Wang theorems. There are natural geometric interpretations for each of the assumptions in these results which explains their role clearly. However, Monge's original cost function $c(x,y)=\|x-y\|$ does not satisfy (NonDeg) or (Twist), and its MTW tensor is undefined (since one cannot invert the mixed Hessian). As might be expected, optimal transport with respect to the cost can be fairly pathological. Transport plans are generally non-unique and need not be induced by a transport map. Nonetheless, it seems possible to recover a partial regularity theory, which is a longstanding question in optimal transport.

 For smooth enough measures, the optimal transport plan for the Monge cost will occur along disjoint line segments (called transport rays). However, one can rearrange the transport on each ray to get another plan with the same total cost. As such, the optimal solution is non-unique and may be non-deterministic (i.e., not induced by a transport map). However, there is a natural assumption which resolves these issues; that the transport is monotonic along transport rays. Doing so, there is a unique solution which is induced by a transport map, known as the \emph{ray monotone solution}. 

For measures in $\mathbb{R}^2$, the monotone optimal mapping will be continuous in the interior of the transfer set (i.e., the union of all transfer rays) when the densities are positive, continuous, and have compact, convex and disjoint supports \cite{FragalaGelliPratelli} (see also, \cite{LiSantWang2}). On the other hand, even among smooth measures with convex supports, it is possible to find examples where the ray monotone solution fails to be Lipschitz \cite{LiSantWang}. As such, there are many questions about the Monge cost which remain unsolved.

 Monge's potential $\Psi_0(r) = r$ is highly singular so the flow defined by \ref{Psicostflow} is not well-defined at time $t=0$. However, using Example \ref{Example1} and taking formal limits as $C$ goes to zero, we can define a new potential obtained by deforming Monge's cost in a canonical way. This deformation which satisfies (MTW) far enough away from the origin.

We start by computing the Ricci potential $\rho(r)$ for Example \ref{Example1}. As the metric is rotationally symmetric, this reduces to computing \[ \rho(r) = \log \left( \det[ \Hess \Psi(x,t)] \right). \]
Using Mathematica, we find the following:
\begin{equation} \rho(r) =
\log \left[\frac{C}{\left(C+r\right)^3}\right].
\end{equation}

It is not possible to take the limit of this potential as $C$ goes to $0$, as it goes to $-\infty$ for $r\neq 0$. Instead, we consider the renormalized Ricci potential $\widetilde \rho(r)$, defined as 
\begin{equation} \tilde \rho(r) = \rho(r)-\log(C) =
\log \left[\frac{1}{\left(C+r\right)^3}\right].
\end{equation}

From the perspective of the parabolic flow \ref{Psicostflow}, renormalizing the Ricci potential acts to slow the flow but does not change the geometry of the solutions otherwise. Now we can take the limit of this potential as $C$ goes to $0$. Doing so, we obtain that
\begin{equation}
    \widetilde \rho_0(r) = -3\log(r).
\end{equation}

In a formal sense, this quantity can be understood as the ``renormalized Ricci potential" associated to the Monge potential $\Psi(r) = r$. Using this quantity, we can construct a formal first-order Taylor series of Equation \ref{Psicostflow}. Doing so, we obtain the potentials
\begin{equation}\label{Monge flow}
    \Psi_{Monge}(r,t)= r- 3t\log(r).
\end{equation}

In the region where $r>3t$, this potential is strictly convex, so defines a metric. Furthermore, in dimension $2$, the associated costs satisfy (MTW) when $r$ is sufficiently large compared to $t$ (for instance, $r> \frac{26}{3}t$ suffices). As such, this gives a one-parameter family of cost functions with good regularity properties far enough from the origin.
 
At this point, the potential $\Psi_{Monge}(r,t)$ might appear somewhat arbitrary, in that there were several seemingly ad hoc choices in its construction. In fact, this potential is not arbitrary, but will appear as fairly universal limit, modulo a gauge transformation. For instance, if one instead uses the deformation $\Psi_\epsilon = \sqrt{\epsilon + r^2}$ instead of Example \ref{Example1}, the renormalized Ricci potential is \begin{equation}
    \widetilde \rho_0(r) = -2\log(r),
\end{equation}
which differs from the original by a multiplicative factor alone. And thus the one-parameter family just differs from the original by a rescaling of the time parameter.\footnote{For other deformations, the same phenomena occurs, although the factor used to renormalize the Ricci potential depends on the deformation.} Furthermore, the use of the first-order Taylor series in $\Psi_{Monge}$ is not arbitrary. By scaling $r$ to be large, we can renormalize (i.e., slow down) the flow even further to ensure that the first-order approximation is arbitrarily accurate.

At present, we are not able to directly apply this analysis to the regularity problem of the Monge cost. However, it may be the case that the limits of the transport map as $t$ goes to $0$ are reasonably well-controlled. This approach was considered in \cite{LiSantWang} using $c_\epsilon = \sqrt{\epsilon + \|x-y\|^2}$, where the authors established that the eigenvalues of the Jacobian matrix are locally uniformly bounded. Going further, one potential route to establishing regularity would be to show that the transports converge in $C^\alpha$ for $\alpha$ small enough.


\section{Acknowledgements}

The first author would like to thank Jun Kitagawa for his helpful comments and discussions.

\newpage

\appendix
\section{Derivation of Example \ref{ComplicatedExample}}

\label{Example2}

In this section, we prove that Example \ref{ComplicatedExample} on page 13 has non-negative anti-bisectional curvature.

Recall that in this example, we set
$$ \ell = r + \frac{1}{L^2}, \ \ \ L = \log (c+r) \textrm{ for $c \log(c)^3 \geq 2$.} $$

Here, $\ell=\frac{r}{\phi^\prime(r)}$ where $\phi$ is the convex function satisfying $\Psi(x) = \phi( |x$). A tedious but straightforward computation shows the following. Note that in the following $\log[c+r]^k$ denotes $\left( \log[c+r] \right)^k$, rather than  $ \log\left(c+r\right)^k$. We omit the extra parenthesis for brevity.

\[ A = \frac{\log[c + r]^2 (-2 + (c + r) \log[c + r]^3)}{r (2 r + (c + r) \log[c + r]) (1 + r \log[c + r]^2)} \]

\[ B = \frac{\log[c + r] (6 r + 
   2 (c + 2 r) \log[c + r] - (c + r)^2 \log[c + r]^4)}{r (c + 
   r) (2 r + (c + r) \log[c + r]) (1 + r \log[c + r]^2)} \]

\[ C = 
\frac{ \left( \begin{aligned}
      -60 r^3 - 4 r^2 (19 c + 26 r) \log[c + r] - 
   2 r (17 c^2 + 49 c r + 2 r^2 (17 + 3 r)) \log[c + r]^2 \\ -
   2 (3 c^3 + 12 c^2 r + 4 r^3 (2 + 3 r) + c r^2 (17 + 6 r)) \log[c + r]^3 - 2 r^2 (5 c^2 + 17 c r + 14 r^2) \log[c + r]^4 \\
   + 
   2 r (2 c^3 + 3 c^2 r - 2 c r^2 - 3 r^3) \log[c + r]^5 + 
   3 (c + r)^2 (c^2 + 2 c r + r^2 - 4 r^3) \log[c + r]^6 \\
    - 
   2 r^2 (c + r)^2 (c + 3 r) \log[c + r]^7 + 
   2 r (c + r)^4 \log[c + r]^8 + 
   r^2 (c + r)^4 \log[c + r]^{10}
      \end{aligned} \right)}%
  {r (c + r)^3 \log[
    c + r] (2 r + (c + r) \log[c + r]) (1 + r \log[c + r]^2)^3}
\]

\[  D = \frac{\left( \begin{aligned}
-60 r^3 - 4 r^2 (15 c + 22 r) \log[c + r] - 
   2 r (9 c^2 + 29 c r + 2 r^2 (11 + 3 r)) \log[c + r]^2 \\
- 2 (3 c^3 + 8 c^2 r + 3 c (3 - 2 r) r^2 + 4 r^3) \log[c + r]^3 + 
   2 r^2 (15 c^2 + 27 c r + 10 r^2) \log[c + r]^4 \\
 + 2 r (6 c^3 + 21 c^2 r + 22 c r^2 + 7 r^3) \log[c + r]^5 + 
   3 (c + r)^4 \log[c + r]^6
 \end{aligned} \right) }{r (c + r)^3 \log[
    c + r] (2 r + (c + r) \log[c + r]) (1 + r \log[c + r]^2)^3}
\]

Whenever $c \log[c]^3 \geq 2$, we have that $A$ and $A+B$ are greater than 0, so what remains to show is that $D \geq 0$.

We start by consider the denominator of $D$
\[Denom[D]= r (c + r)^3 \log[c + r] (2 r + (c + r) \log[c + r]) (1 + r \log[c + r]^2)^3. \]

This is positive for $r \geq 0$ whenever $c\geq 1$, so the denominator is positive. As such, what remains to show is that the numerator is also positive.

\[ Numer[D] =
\left( \begin{aligned}
-60 r^3 - 4 r^2 (15 c + 22 r) \log[c + r] - 
   2 r (9 c^2 + 29 c r + 2 r^2 (11 + 3 r)) \log[c + r]^2 \\
- 2 (3 c^3 + 8 c^2 r + 3 c (3 - 2 r) r^2 + 4 r^3) \log[c + r]^3 + 
   2 r^2 (15 c^2 + 27 c r + 10 r^2) \log[c + r]^4 \\
 + 2 r (6 c^3 + 21 c^2 r + 22 c r^2 + 7 r^3) \log[c + r]^5 + 
   3 (c + r)^4 \log[c + r]^6
 \end{aligned} \right)
\]

Momentarily treating powers of $\log[c+r]$ as if they were constants, then this appears to be a quartic polynomial in $r$. Arranging the terms in this fashion, we find
\[ Numer[D] = \mathcal{A}_0 + \mathcal{A}_1 r + \mathcal{A}_2 r^2 + \mathcal{A}_3 r^3 + \mathcal{A}_4 r^4, \textrm{ where,}
\]
\begin{eqnarray*}
\mathcal{A}_0 &=&  3 c^3 \log[c+r]^3 (-2 + c \log[c+r]^3), \\
\mathcal{A}_1 &= & -18 c^2 \log[c + r]^2 - 16 c^2 \log[c + r]^3 + 12 c^3 \log[c + r]^5 + 
 12 c^3 \log[c + r]^6, \\
 \mathcal{A}_2 & = & -60 c \log[c + r] - 58 c \log[c + r]^2 - 18 c \log[c + r]^3 \\
 & & +  30 c^2 \log[c + r]^4 + 42 c^2 \log[c + r]^5 + 18 c^2 \log[c + r]^6, \\
 \mathcal{A}_3 &= & -60 - 88 \log[c + r] - 44 \log[c + r]^2 - 8 \log[c + r]^3  \\
 & & + 12 c \log[c + r]^3 + 54 c \log[c + r]^4 + 44 c \log[c + r]^5 + 
 12 c \log[c + r]^6, \\
\mathcal{A}_4 &=& -12 \log[c + r]^2 + 20 \log[c + r]^4 + 14 \log[c + r]^5 + 3 \log[c + r]^6.
\end{eqnarray*}

We will now show that each of these terms are non-negative, which completes the proof that $D \geq 0$.

\begin{enumerate}
    \item [$\mathcal{A}_0 \geq 0$.]
    
   Note that $c^3 \log[c+r]^3 >0$ and we have assumed that $c \log[c]^3 \geq 2$, so both factors are non-negative. As such, $\mathcal{A}_0 \geq 0$.
    \item [$\mathcal{A}_1>0$.]
    
    Divide $\mathcal{A}_1$ by $ c^2 \log[c+r]^2 >0$ to obtain
    \[  \mathcal{A}^\prime_1 = 2 (-9 - 8 \log[c + r] + 6 c \log[c + r]^3 + 6 c \log[c + r]^4). \]
    
    By the assumption that $c \log[c]^3 \geq 2$, we have that 
     \[  \mathcal{A}^\prime_1 \geq 2 (-9 - 8 \log[c + r] + 12 + 12 \log[c + r]) = 3 + 4 \log[c+r] >0. \]
     As such, we have that $\mathcal{A}_1 > 0$.
     \item [$\mathcal{A}_2>0$.]
     
     Simplifying $\mathcal{A}_2$ by dividing out $c \log[c+r]$, we obtain 
     \[ \frac{\mathcal{A}_2}{c \log[c+r]} = 2 \left( \begin{aligned} -30 - 29 \log[c + r] - 9 \log[c + r]^2 + 15 c \log[c + r]^3 \\ + 
   21 c \log[c + r]^4 + 9 c \log[c + r]^5 \end{aligned} \right). \]
   Once again using our assumption on $c$, we have that 
  \begin{eqnarray*}
  \frac{\mathcal{A}_2}{c \log[c+r]} & \geq & 2 \left( - 29 \log[c + r] - 9 \log[c + r]^2 + 
   42 \log[c + r] + 18 \log[c + r]^2\right) \\
   &= &2( 13 \log[c + r] + 9 \log[c + r]^2) >0  
   \end{eqnarray*}
   As such, $\mathcal{A}_2>0$. 
   \item [$\mathcal{A}_3>0$.]
   
   To show that this term is positive, we cannot simply bound each term by below, as we did for the previous terms. Instead, consider $\mathcal{A}_3(r,c)$ as a function of $r$ and $c$. We first show that when $c$ satisfies $c \log[c]^3 \geq 2$, $\mathcal{A}_3(0,c) > 0$. To see this, observe that
   \begin{eqnarray*} \mathcal{A}_3(0,c)   &= & -60 - 88 \log[c] - 44 \log[c]^2 - 8 \log[c]^3 + 12 c \log[c]^3 \\
   &  & +  54 c \log[c]^4 + 44 c \log[c]^5 + 12 c \log[c]^6 \end{eqnarray*}
    
    We estimate this term from below as follows. 
   \begin{eqnarray*}
   \mathcal{A}_3(0,c) &\geq & -60 - 88 \log[c] - 44 \log[c]^2 - 8 \log[c]^3 \\
   & & + 24 +  108 \log[c] + 88 \log[c]^2 + 24 \log[c]^3 \\
    & = & -36 + 20 \log[c] + 44 \log[c]^2 + 16 \log[c]^3
   \end{eqnarray*} 
    
From the fact that $c \log[c]^3 \geq 2,$ we have that $\log[c] > .87$.\footnote{Note that
$.87^3 \exp[.87] < .87^3 e <.87^3 \cdot 3 = 1.975509 < 2, $ which is how we obtain this lower bound. }
Plugging in $.87$ as a lower bound for $\log[c]$ in the previous inequality, we have that 
    \[  \mathcal{A}_3(0,c) > -36 + 20 \cdot .87 + 44 \cdot (.87)^2 + 16 \cdot (.87)^3 = 25.239648 >0.  \] 

 Finally, we show that $(r+c) \frac{\partial}{\partial r} \mathcal{A}_3(r,c) >0$.
         
       \begin{eqnarray*}
       (r+c) \frac{\partial}{\partial r}\mathcal{A}_3(r,c) &= &4 \left( \begin{aligned}-22 - 22 \log[c + r] + (-6 + 9 c) \log[c + r]^2   \\
    +  54 c \log[c + r]^3    + 55 c \log[c + r]^4 + 18 c \log[c + r]^5\end{aligned} \right) \\
       & \geq & 4 \left( \begin{aligned}-22 - 22 \log[c + r] + (-6 + 9 c) \log[c + r]^2  \\
    +  108   + 110 \log[c + r] + 36 \log[c + r]^2 \end{aligned} \right) \\ 
       & = & 4 (86 +88 \log[c + r] + (27 + 9 c) \log[c + r]^2) > 0.
   \end{eqnarray*}
This implies that $\frac{\partial}{\partial r}\mathcal{A}_3(r,c) >0$, which further implies that $\mathcal{A}_3(r,c)> \mathcal{A}_3(0,c)$ for all $r>0$. Since $\mathcal{A}_3(0,c) > 0$, this implies that $\mathcal{A}_3>0$.

    \item [$\mathcal{A}_4>0$.]
   
   Simplifying $\mathcal{A}_4$ by dividing out $\log[c+r]^2$, we  obtain
 \[ \mathcal{A}^\prime_4 = -12 + 20 \log[c + r]^2 + 14 \log[c + r]^3 + 3 \log[c + r]^4 .\]
   
Now, using the fact that $\log[c+r] > .87$ (as shown above), we have that
  \begin{eqnarray*}
    \mathcal{A}^\prime_4 & > & -12 + 20 (.87)^2 + 14 (.87)^3 + 3 (.87)^4 \\
    & = & 14.07573483 >0
   \end{eqnarray*}
\end{enumerate}

Since all these terms are non-negative, this implies that 
\[Numer[D] = \mathcal{A}_0 + \mathcal{A}_1 r + \mathcal{A}_2 r^2 + \mathcal{A}_3 r^3 + \mathcal{A}_4 r^4 \geq 0. \] Since the denominator of $D$ is also positive, this implies that $D>0$, and so this metric has non-negative anti-bisectional curvature.

\end{document}